\documentclass[a4paper,11pt]{amsart}

\usepackage{a4wide} 

\usepackage{color}
\usepackage{hyperref}
\usepackage{tensor}
\usepackage{enumitem} 
\usepackage{amsfonts,amssymb,amsmath,amsthm,graphicx,latexsym}
\usepackage{mathrsfs,dsfont}
\usepackage{pgfplots}
\usepackage{tikz}
\usepackage{bm}
\usepackage{caption}

\newcommand{\R}{\mathbb R}
\newcommand{\C}{\mathbb C}
\newcommand{\N}{\mathbb N}
\newcommand{\Z}{\mathbb Z}

\newtheorem{thm}{Theorem}[section]
\newtheorem{prop}[thm]{Proposition}
\newtheorem{lemma}[thm]{Lemma}
\newtheorem{cor}[thm]{Corollary}
\newtheorem{rem}[thm]{Remark}
\newtheorem{exm}[thm]{Examples}
\newtheorem{df}[thm]{Definition}

\newcommand{\tp}{\mathbin{\hbox{$\bigcirc$\hbox to
      0pt{\hspace{-0.81em}$\scriptstyle\top$\hfil}}}}

\begin{document}

\title[$(p, q)$-Touchard polynomials and  Stirling numbers]%
  {$(p, q)$-analogues of the generalized Touchard polynomials and  Stirling numbers}
\author{Lahcen Oussi}
\address{Institute of Mathematics, University of Wroc\l aw, Poland.}
\email{lahcen.oussi@math.uni.wroc.pl}

\begin{abstract}
In this paper we introduce a $(p, q)$-deformed analogues of the  generalized Touchard polynomials and Stirling numbers, the post-quantum analogues of the $q$-deformed generalized Touchard polynomials and Stirling numbers. The connection between these deformations is established. A recurrence relation for the $(p, q)$-deformed generalized Touchard polynomials is expounded, elucidating a $(p, q)$-deformation of Spivey's relation.
\end{abstract}
\maketitle

\section{Introduction}\label{intro}
The Stirling numbers of the second kind, denoted by $S(n, k)$, count the number of ways to partition a set of size $n$ into $k$ non-empty disjoint subsets. They appear as coefficients in the normal ordering of $\left(X\frac{d}{dx}\right)^n$ as follows
\begin{equation}\label{snor}
\left(X\frac{d}{dx}\right)^n=\sum_{k=0}^{n}S(n, k)X^k\left(\frac{d}{dx}\right)^k,
\end{equation}
where $X$ denotes the operator of multiplication with the variable, that is,  $Xf(x)=xf(x)$, and $\frac{d}{dx}$ is the derivative operator with respect to $x$.
The relation \eqref{snor} has deep connection in the physical literature, which was rediscovered by Katriel \cite{K1947}. Namely, it is connected with the creation operator $a$ and the annihilation operator $a^{\dagger}$ satisfying the commutation relation $aa^{\dagger}-a^{\dagger}a=1$ in the Boson Fock space \cite{MADC, MMM}.\\
Moreover, the Stirling numbers of the second kind $S(n, k)$ appear also as coefficients in the expansion of 
$$x^n=\sum_{k=0}^{n}S(n, k)\prod_{j=0}^{k-1}(x-j),$$
and they satisfy the recurrence formula 
\begin{equation}\label{rcs}
S(n, k)=S(n-1, k-1)+kS(n-1, k).
\end{equation}
The Bell numbers $B_{n}$ are known by the following relation 
$$B_{n}=\sum_{k=0}^{n}S(n, k),$$
and they satisfy the recursive formula
\begin{equation}\label{rbn}
B_{n+1}=\sum_{k=0}^{n} \binom{n}{k} B_{k}.
\end{equation}
In 2008, Spivey \cite{MZS} generalized the relation \eqref{rbn} as follows 
\begin{equation}\label{csp}
B_{n+m}=\sum_{j=0}^{m}\sum_{k=0}^{n}k^{m-j}S(n, k)\binom{m}{j}B_{j}.
\end{equation}
The above formula is known in the literature as ``Spivey's Bell number formula''.\\
Subsequently, in the same year, Katriel \cite{JK2008} derived a $q$-deformed formula of \eqref{csp}:
\begin{equation}\label{qsp}
B_{n+m}(q)=\sum_{j=0}^{m}\sum_{k=0}^{n}[k]_{q}^{m-j}q^{jk}S_{q}(n, k)\binom{m}{j}B_{j}(q),
\end{equation}
where 
\begin{equation}\label{qstr}
S_{q}(n, k)=q^{k-1}S(n-1, k-1)+[k]_{q}S_{q}(n-1, k),
\end{equation}
are the $q$-Stirling numbers of the second kind with the initial value $S_{q}(0, 0)=1$, and $[k]_{q}:=\frac{1-q^k}{1-q}$ denotes the $q$-number \cite{FHJ1909, FHJ19010, VKPC2002}. Consequently, the associated $q$-Bell numbers are given by 
$$B_{n}(q)=\sum_{k=0}^{n}S_{q}(n, k).$$
Recently, the present author \cite{LO} derived a $(p, q)$-deformed variant of Eq. \eqref{csp}. The resulting relation can be expressed as follows
\begin{equation}\label{pqs}
B_{n+m}(p, q; 1)=\sum_{k=0}^{n}\sum_{j=0}^{m}p^{\binom{k}{2}}[k]_{p, q}^{m-j}q^{jk}S_{p, q}(n, k)\binom{m}{j}B_{j}(p, q; p^{n+m-j}),
\end{equation}
where
\begin{equation}\label{pqstr}
S_{p, q}(n, k)=p^{n-k}q^{k-1}S(n-1, k-1)+[k]_{p, q}S_{p, q}(n-1, k),
\end{equation}
are the $(p, q)$-Stirling numbers of the second kind, with $S_{p, q}(0, 0)=1$, $[k]_{p, q}:=\frac{p^k-q^k}{p-q}$ denotes the $(p, q)$-number, which reduces to $[k]_{q}$ for $p=1$, and $B_{n}(p, q; 1)=\sum_{k=0}^{n}p^{\binom{k}{2}}S_{p, q}(n, k)$ are the corresponding $(p, q)$-Bell numbers in analogy to the classical and deformed cases.\\
Obviously, for $p=1$, Eqs. \eqref{pqs} and \eqref{pqstr} reduce to \eqref{qsp} and \eqref{qstr}, respectively. If further $q=1$, they reduce to \eqref{csp} and \eqref{rcs}, respectively.\\
It is worthwhile to mention that various generalizations of Stirling and Bell numbers have been extended in several ways by several authors, see for instance \cite{PBKAP2003, KNB2, LC1932, BSENPCTM2010, OHTM2017, FTH1985, WL2009, WL, TMMSMS2011, TMMSMS, MAMPBKAP2005, MS2003}. In particular, Mansour et al. considered in \cite{TMMSMS2011} the noncommutative operators $U$ and $V$ satisfying the commutation relation $UV=VU+hV^{s}$ (for $s\in{\R}$ and $h\in{\C}\setminus \{0\}$), and introduced a new family of generalized Stirling numbers $\mathfrak{S}_{s; h}(n, k)$ by 
\begin{equation}\label{newgsn}
\left(VU\right)^{n}=\sum_{k=1}^{n}\mathfrak{S}_{s; h}(n, k)V^{s(n-k)+k}U^{k}.
\end{equation}
Moreover, they introduced in \cite{TMMSMS} the $q$-version of the above generalized Stirling numbers, which appear in the normal ordering of $(VU)^n$ of the noncommutative operators $U$ and $V$ satisfy the $q$-commutation relation $UV=qVU+hV^{s}$. Namely
\begin{equation}\label{qgstr1}
\left(VU\right)^n=\sum_{k=1}^{n}\mathfrak{S}_{s; h}(n, k|q)V^{s(n-k)+k}U^{k},
\end{equation}
and 
\begin{equation}\label{qgstirling}
\mathfrak{S}_{s; h}(n+1, k|q)=q^{s(n+1-k)+k-1}\mathfrak{S}_{s; h}(n, k-1|q)+h[s(n-k)+k]_{q}\mathfrak{S}_{s; h}(n, k|q)
\end{equation} 
for all $n\geq 0$, $k\geq 1$, with the initial values $\mathfrak{S}_{s; h}(n, 0|q)=\delta_{n, 0}$ and $\mathfrak{S}_{s; h}(0, k|q)=\delta_{0, k}$ for all $n, k\in{\N}_{0}:=\{0, 1, 2, \ldots\}$, with the corresponding generalized Bell numbers 
\begin{equation}\label{qgbell}
\mathfrak{B}_{s; h}(n|q):=\sum_{k=0}^{n}\mathfrak{S}_{s; h}(n, k|q). 
\end{equation}
In 2015, Mansour and Schork \cite{TMMS} derived the following  $q$-deformed generalization of Spivey's relation \eqref{csp}
\begin{equation}\label{qgsp}
\begin{split}
\mathfrak{B}_{\frac{m-1}{m}; [m]_ {q}}(n+l|q^m) &=\sum_{k=0}^{n}\sum_{j=0}^{l}\bigg(
\begin{bmatrix}
n \\
k 
\end{bmatrix}_{q^{m-1}, h_{m, j+(m-1)l}(q)}\\
&\times [j+(m-1)l]_{q}^{n-k}\mathfrak{S}_{\frac{m-1}{m}; [m]_{q}}(l, j|q^m)q^{k(j+(m-1)l)}\mathfrak{B}_{\frac{m-1}{m}; [m]_{q}}(k| q^m)\bigg),
\end{split}
\end{equation}
where $h_{m, s}(q):=q^s\frac{[m-1]_{q}}{[s]_{q}}$ and $
\begin{bmatrix}
n \\
k 
\end{bmatrix}_{(q, h)}:=
\begin{bmatrix}
n \\
k 
\end{bmatrix}_{q}\prod_{i=0}^{k-1}(1+h[i]_{q})$ is the $(q, h)$-binomial coefficients \cite{HBB}, with $\begin{bmatrix}
n \\
k 
\end{bmatrix}_{q}:=\frac{[n]_{q}!}{[k]_{q}![n-k]_{q}!}$ denotes the $q$-Gauss binomial coefficients.\\
In the present work (Section \ref{sec3}), we derive the $(p, q)$-analogues of the  generalized Stirling numbers  and Spivey's relation above. 

The Touchard polynomials (also known as exponential polynomials or Bell polynomials) have many applications and play important rule in areas of probability and statistics, see e.g., \cite{KNB1, LC1979, DSKTK2015, DSKTKTMJJS2016, TMMS2015, TMSMMS2012}. In particular, the Touchard polynomial $T_{n}(\lambda)$ is the $n$-th moment of a random variable $Y$ which has Poisson distribution and expected value $\lambda$.\\
In an operational framework, the Touchard polynomials can be defined \cite{KNB1, KNB2, Detal, SR, GCR} by the Rodrigues-like formula:
\begin{equation}\label{ctchrd}
T_{n}(x)=e^{-x}\left(X\frac{d}{dx}\right)^ne^x\quad \text{ for }n\in{\N}_{0}.
\end{equation}
In 2012, Dattoli et al. \cite{Detal} introduced a generalization of the Touchard polynomials \eqref{ctchrd} for higher order $m\in{\N}_{0}$ as follows:
\begin{equation}\label{gtchrd}
T_{n}^{(m)}(x)=e^{-x}\left(X^m\frac{d}{dx}\right)^ne^x, \text{ for } n\in{\N}_{0},
\end{equation}
which reduce to the classical Touchard polynomials \eqref{ctchrd} for $m=1$. Moreover, they established the following recursive formula:
\begin{equation}\label{rctch}
T_{n+1}^{(m)}(x)=x^m\left(1+\frac{d}{dx}\right)T_{n}^{(m)}(x).
\end{equation}
Mansour et al. extended the generalized Touchard polynomials for negative integer order $m\in{\Z}$ in \cite{TMMSMS}, and furthermore, for arbitrary real order $m\in{\R}$ in \cite{TMMS2013}. On the other hand, Mansour and Schork \cite{TMMS} introduced a $q$-deformation of the generalized Touchard polynomials for $m\in{\R}$ by the following 
\begin{equation}\label{qgtchrd}
T_{n; q}^{(m)}(x)=E_{q}(-x)(X^mD_{q})^ne_{q}(x),
\end{equation}
where $D_{q}$ denotes the analogue of the derivative operator $\frac{d}{dx}$ (also called Jackson derivative), given by $D_{q}f(x)=\frac{f(x)-f(qx)}{1-q}$, $E_{q}(x)=\sum_{n\geq 0}\frac{q^{\binom{2}{n}}x^n}{[n]_{q}!}$ and $e_{q}(x)=\sum_{n\geq 0}\frac{x^n}{[n]_{q}!}$ are the $q$-exponential functions, and $[n]_{q}!=\prod_{j=1}^{n}[j]_{q}$ is the $q$-factorial function of the $q$-number $[n]_{q}$.\\
The $q$-deformed generalized Touchard polynomials satisfy the recursion formula \cite{TMMS}:
\begin{equation}\label{rqtch}
T_{n; q}^{(m)}(x)=x^m\Big(1+E_{q}(-x)e_{q}(qx)D_{q}\Big)T_{n; q}^{(m)}(x),
\end{equation}
which reduces to \eqref{rctch} for q=1.\\
Using the normal ordering \eqref{snor}
and the fact that (see \cite{TMMS})
\begin{equation*}
(X^mD_{q})^n=X^{n(m-1)}\sum_{k=0}^{n}\mathfrak{S}_{\frac{m-1}{m}; [m]_{q}}(n, k|q^m)X^kD_{q}^k,
\end{equation*}
one obtains directly from \eqref{ctchrd} and \eqref{qgtchrd}
the following relationships between the Stirling numbers $S(n, k)$ (resp. $q$-generalized Stirling numbers $\mathfrak{S}_{s; h}(n, k|q)$) and the Touchard polynomials $T_{n}(x)$ (resp. $q$-generalized Touchard polynomials $T_{n; q}^{(m)}(x)$) :
\begin{equation}\label{noctchrd}
T_{n}(x)=\sum_{k=0}^{n}S(n, k)x^k=:B_{n}(x) ,
\end{equation}
and 
\begin{equation}\label{nogtchrd}
T_{n; q}^{(m)}(x)=x^{n(m-1)}\sum_{k=0}^{n}\mathfrak{S}_{\frac{m-1}{m}; [m]_{q}}(n, k|q^m)x^k=x^{n(m-1)}\mathfrak{B}_{\frac{m-1}{m}; [m]_{q}|n; q^m}(x), 
\end{equation}
where $B_{n}(x)$ and $\mathfrak{B}_{s; h|n; q}(x):=\sum_{k=0}^{n}\mathfrak{S}_{s; h}(n, k|q)x^k$ are the associated Bell polynomials and $q$-generalized Bell polynomials, respectively.\\
Noting that for $q=1$, the undeformed generalized Touchard polynomials $T_{n}^{(m)}(x)$ satisfy 
$$T_{n}^{(m)}(x)=x^{n(m-1)}\sum_{k=0}^{n}S_{m, 1}(n, k)x^k,$$
where, $S_{m, 1}(n, k)=\mathfrak{S}_{\frac{m-1}{m}; m}(n, k|1)$ are another version of the generalized Stirling numbers of the second kind considered for instance by Lang \cite{WL}.

The structure of the paper is outlined as follows:\\
In Section \ref{sec2} we recall and fix some elementary notions and terminology about $(p, q)$-calculus, which are useful for the rest of the paper. Section \ref{sec3} contains the main result and is divided into two subsections. In Subsection \ref{subsec1}, the $(p, q)$-deformed generalized Stirling numbers of the second kind are introduced, the case $p=1$ corresponds to the $q$-version. In Subsection \ref{subsec2}, we discuss the $(p, q)$-deformed generalized Touchard polynomials for arbitrary real order and their relation to the $ (p, q)$-deformed generalized Stirling numbers of the second kind. A recurrence relation for these polynomials and some other properties are derived as well as the $(p, q)$-deformed generalized Spivey formula and the $(p, q)$-Dobiński formula. Finally, in Section \ref{sec4} we outline some remarks and mention some possible open questions related to this study for future research. 
\section{Preliminaries}\label{sec2} 
In this section we review some basic tools and terminology about $(p, q)$-calculus which we will use in the subsequent sections. For more details about $(p, q)$-calculus and related aspects, we refer the reader to \cite{CR, GV, GVRTMAPNAAM, PNS2018}.\\
For $x\in{\R}$ with $p\neq q$, the $(p, q)$-number (or twin-basic number) $[x]_{p, q}$ is defined as
\begin{equation}\label{pqnumber}
[x]_{p, q}:=\frac{p^x-q^x}{p-q}.
\end{equation} 
If $n$ is non-negative integer, then

\begin{displaymath}
[n]_{p, q} := \begin{cases}
\frac{p^n-q^n}{p-q}=\sum\limits_{k=1}^{n}p^{n-k}q^{k-1}, & \text{if $n\in{\N}:=\{1, 2, 3, \ldots\}$;}\\
0, & \text{if $n=0$ .}
\end{cases}
\end{displaymath}
The $(p, q)$-factorial is given by 
\begin{displaymath}
[n]_{p, q}!:= \begin{cases}
\prod\limits_{k=1}^{n}[k]_{p, q}, & \text{if $n\in{\N}$;}\\
1, & \text{if $n=0$ ,}
\end{cases}
\end{displaymath}
and the $(p, q)$-Gaussian binomial coefficient is defined by 
\begin{displaymath}
\begin{bmatrix}
n \\
k 
\end{bmatrix}_{p, q}:= \begin{cases}
\frac{[n]_{p, q}!}{[n-k]_{p, q}![k]_{p, q}!}, & \text{if $0\leq k \leq n$;}\\
0, & \text{if $k>n\geq 0$ or $k<0$ .}
\end{cases}
\end{displaymath}
Moreover, we define the following operators: 
\begin{enumerate}
\item The $(p, q)$-derivative operator $D_{p, q}$:
\begin{displaymath}
D_{p, q}f(x)=\begin{cases}
\frac{f(px)-f(qx)}{x(p-q)}, & \text{if $x\neq 0$;}\\
f'(0), & \text{if $x=0$,}
\end{cases}
\end{displaymath}
for functions $f$ which are differentiable at $x=0$.
\item The operator of multiplication $X$ with the variable:
$$Xf(x)=xf(x)$$
\item The Fibonacci operator $N_{p}$ \cite{JCMDCH}:
$$N_{p}f(x)=f(px).$$

\end{enumerate}
It is clear that $D_{p, q}x^k=[k]_{p, q}x^{k-1}$, and the $(p, q)$-Leibniz formula reads
\begin{equation}\label{pqlf} 
D_{p, q}^{n}(f(x)g(x))=\sum_{k=0}^{n}\begin{bmatrix}
n \\
k 
\end{bmatrix}_{p, q}D_{p, q}^{n-k}(f)(p^kx)D_{p, q}^{k}(g)(q^{n-k}x),
\end{equation}
where, $D_{p, q}^{n}f(x):=D_{p, q}(D_{p, q}^{n-1}f(x))$ for $n=1, 2, 3, \ldots$, and $D_{p, q}^{0}$ denotes the identity operator.\\ 
The operators $D_{p, q}, X$ and $N_{p}$ satisfy the following $(p, q)$-commutation relations:
\begin{equation}\label{cr1}
D_{p, q}X=qXD_ {p, q}+N_{p}, 
\end{equation}
\begin{equation}\label{cr2}
N_{p}X=pXN_{p},
\end{equation}
and 
\begin{equation}\label{cr3}
D_{p, q}N_{p}=pN_{p}D_{p, q}.
\end{equation}
The $(p, q)$-deformed exponential functions $E_{p, q}(x)$ and $e_{p, q}(x)$ can be  defined, respectively, as follows:
\begin{equation}\label{exp1}
E_{p, q}(x)=\sum_{n\geq 0}\frac{q^{\binom{n}{2}}}{[n]_{p, q}!}x^n,
\end{equation}
and 
\begin{equation}\label{exp2}
e_{p, q}(x)=\sum_{n\geq 0}\frac{p^{\binom{n}{2}}}{[n]_{p, q}!}x^n, 
\end{equation}
which satisfy the basic identity
\begin{equation}\label{expid}
e_{p, q}(x)E_{p, q}(-x)=1.
\end{equation}
The $n$th-$(p, q)$-derivative of the exponential function $e_{p, q}(x)$ satisfies the differential equation \cite{PNS}:
\begin{equation}\label{pqdex}
D_{p, q}^{n}e_{p, q}(x)=p^{\binom{n}{2}}e_{p, q}(p^nx).
\end{equation}
The $q$-settings (which correspond to $p=1$) of the above notions are the ones considered in \cite{TE2003, FHJ1909, VKPC2002, TMMS, LO}.\\
We conclude this section by recalling the following proposition from \cite{LO} which will be useful in the subsequent computations.
\begin{prop}[\cite{LO}]\label{propo}
Let $n$ be a nonnegative integer. Then the following relation holds true:
\begin{equation}\label{meq}
(XD_{p, q})^n=\sum_{k=0}^{n}S_{p, q}(n, k)X^kN_{p}^{n-k}D_{p, q}^k,
\end{equation}
where
\begin{equation}
S_{p, q}(n, k)=p^{n-k}q^{k-1}S_{p, q}(n-1, k-1)+[k]_{p, q}S_{p, q}(n-1, k),
\end{equation}
are the $(p, q)$-Stirling numbers of the second kind, with the initial value $S_{p, q}(0, 0)=1$, and consequently the numbers
\begin{equation}
\label{pqbell}
B_{n}(p, q)=\sum_{k=0}^{n}S_{p, q}(n, k),
\end{equation}
can be considered as the associated $(p, q)$-Bell numbers.
\end{prop}
\section{Main result}\label{sec3}
\subsection{\texorpdfstring{The $(p, q)$-deformed generalized Stirling numbers of the second kind}{}}\label{subsec1}
In this subsection, we introduce a $(p, q)$-analogue of the generalized Stirling numbers of the second kind which were introduced by Mansour et al. in \cite{TMMSMS2011}.\\
For $h\in\mathbb{C}\setminus\{0\}$ and $s\in\mathbb{R}$, let $U, V$ and $W_{p}$ be noncommutative operators satisfy by the following $(p, q)$-commutation relations:
\begin{align}
UV-qVU&=hV^sW_{p},\label{eq1}\\
W_{p}V&=pVW_{p},\label{eq2}\\
UW_{p}&=pW_{p}U,\label{eq3}\\
W_{1}&=I \quad (\text{ identity operator}).\label{eq4}
\end{align}
The case $p=1$ and $s=0$, reduces to the $q$-commutation relation $UV-qVU=h$ of the $q$-deformed Heisenberg-Weyl algebra. Furthermore, $q=1$ yields the usual Heisenberg-Weyl algebra $([U, V]:=UV-VU=h)$. On the other hand, choosing $p=q=s=1$, one obtains the commutation relation $UV-VU=hV$ of the two-dimensional non-Abelian Lie algebra \cite{RAS1958, RMW1967, WW1974}, and $p=q=1$ and $s=2$ yield the Jordan quantum plane \cite{EEDYIMEEMDVZ1990, YIM1991}.
\begin{prop}\label{propuvw}
For all $k\in{\N}$, we have the following identity:
\begin{equation}\label{eq5}
UV^k-q^kV^kU=h[k]_{p, q}V^{s+k-1}W_{p},
\end{equation}
where $U, V$ and $W_{p}$ are noncommutative operators  satisfying Eqs. \eqref{eq1}-\eqref{eq4}.
\end{prop}
\begin{proof}
The proof will be by induction on $k$. For $k=1$, Eq. \eqref{eq5} holds, since Eq. \eqref{eq1}. Assume that Eq. \eqref{eq5} holds for $k$ and let us prove it for $k+1$. Using induction hypothesis and Eqs. \eqref{eq1} and \eqref{eq2}, we obtain
\begin{align*}
UV^{k+1}&=q^kV^kUV+h[k]_{p, q}V^{s+k-1}W_{p}V\\
&=q^kV^k(qVU+hV^sW_{p})+h[k]_{p, q}V^{s+k-1}W_{p}V\\
&=q^{k+1}V^{k+1}U+hq^kV^{k+s}W_{p}+h[k]_{p, q}V^{s+k-1}W_{p}V\\
&=q^{k+1}V^{k+1}U+hq^kV^{k+s}W_{p}+hp[k]_{p, q}V^{s+k}W_{p}\\
&=q^{k+1}V^{k+1}U+h(q^k+p[k]_{p, q})V^{s+k}W_{p}\\
&=q^{k+1}V^{k+1}U+h[k+1]_{p, q}V^{s+k}W_{p},
\end{align*}
which completes the proof.
\end{proof}
\begin{cor}\label{propdxn}
One can reformulate Proposition \ref{propuvw} in terms of the operators $D_{p, q}, X$ and $N_{p}$ such that $U=D_{p, q}, V=X$ and $W_p=N_{p}$ (with $h=1, s=0$), as follows
\begin{equation}\label{propdxneq}
D_{p, q}X^k=q^kX^kD_{p, q}+[k]_{p, q}X^{k-1}N_{p}.
\end{equation}
\end{cor}
\begin{df}\label{defst}
Let $h\in{\C}\setminus\{0\}$ and $s\in{\R}$. For nonnegative integers $n, k$, we define the $(p, q)$-deformed generalized Stirling numbers of the second kind $\mathfrak{S}_{s; h}(n, k|p, q)$, as normal ordering coefficients of $(VU)^n$ by the following:
\begin{equation}\label{Gstr}
(VU)^n=\sum_{k=0}^{n}\mathfrak{S}_{s; h}(n, k|p, q)V^{s(n-k)+k}W_{p}^{n-k}U^k,
\end{equation}
with $U, V$ and $W_{p}$  satisfy Eqs. \eqref{eq1}-\eqref{eq4}. \\
Moreover, the corresponding $(p, q)$-deformed generalized Bell polynomials  and Bell numbers are defined, respectively, by 
\begin{equation}\label{gbp}
\mathfrak{B}_{s; h|n; p, q}(x):=\sum_{k=0}^{n}\mathfrak{S}_{s;  h}(n, k|p, q)x^k, 
\end{equation}
and 
\begin{equation}\label{gbn}
\mathfrak{B}_{s; h}(n|p, q):=\sum_{k=0}^{n}\mathfrak{S}_{s;  h}(n, k|p, q).
\end{equation}
\end{df}
\begin{exm} 
From Definition \ref{defst} and using Proposition \ref{propuvw} with the commutation relation \eqref{eq3} we can determine directly the first few instances of the $(p, q)$-deformed generalized Stirling numbers as follows:
\begin{description}
\item[$n=1$]
$\mathfrak{S}_{s; h}(1, 1|p, q)=1$ and $\mathfrak{B}_{s; h}(1|p, q)=1$.
\item[$n=2$]
$\mathfrak{S}_{s; h}(2, 1|p, q)=h, \quad \mathfrak{S}_{s; h}(2, 2|p, q)=q$, and consequently, $\mathfrak{B}_{s; h}(2|p, q)=h+q$.
\item[$n=3$]
$\mathfrak{S}_{s; h}(3, 1|p, q)=h^2[s+1]_{p, q},\hspace{0.3cm} \mathfrak{S}_{s; h}(3, 2|p, q)=hq\left([2]_{p, q}+pq^s\right), \hspace{0.3cm}\mathfrak{S}_{s; h}(3, 3|p, q)=q^3$,
and consequently, the associated Bell numbers 
$$\mathfrak{B}_{s; h}(3|p, q)=q^3+hq([2]_{p, q}+q^sp)+h^2[s+1]_{p, q}.$$
\end{description}
\end{exm}
\begin{thm}
The $(p, q)$-deformed generalized Stirling numbers $\mathfrak{S}_{s; h}(n, k|p, q)$ satisfy the following recursive formula: 
\begin{equation}\label{recst}
\mathfrak{S}_{s; h}(n+1, k|p, q)=p^{n-k+1}q^{s(n-k+1)+k-1}\mathfrak{S}_{s; h}(n, k-1|p, q)+h[s(n-k)+k]_{p, q}\mathfrak{S}_{s; h}(n, k|p, q),
\end{equation} 
for all $n\geq 0$ and $k\geq 1$, with the initial conditions $\mathfrak{S}_{s; h}(n, 0|p, q)=\delta_{n, 0}$ for all $n\geq 0$, and $\mathfrak{S}_{s; h}(0, k|p, q)=\delta_{0, k}$ for all $k\geq 0$.
\end{thm}
\begin{proof}
By Definition \ref{defst}, one has
\begin{equation}\label{prdef}
(VU)^{n+1}=\sum_{k=0}^{n+1}\mathfrak{S}_{s; h}(n+1, k|p, q)V^{s(n-k+1)+k}W_{p}^{n-k+1}U ^k.
\end{equation}
On the other hand, 
\begin{align*}
(VU)^{n+1}&=(VU)(VU)^n\\
&=VU\sum_{k=0}^{n}\mathfrak{S}_{s; h}(n, k|p, q)V^{s(n-k)+k}W_{p}^{n-k}U^k\\
&=\sum_{k=0}^{n}\mathfrak{S}_{s; h}(n, k|p, q)V(UV^{s(n-k)+k})W_{p}^{n-k}U^k\\
&=\sum_{k=0}^{n}\mathfrak{S}_{s; h}(n, k|p, q)V\Big(q^{s(n-k)+k}V^{s(n-k)+k}U+h[s(n-k)+k]_{p, q}\\
 &\qquad\times V^{s(n-k)+k+s-1} W_{p}\Big)\times W_{p}^{n-k}U^k\\
&=\sum_{k=0}^{n}\mathfrak{S}_{s; h}(n, k|p, q)\Big(q^{s(n-k)+k}V^{s(n-k)+k+1}UW_{p}^{n-k}U^k+h[s(n-k)+k]_{p, q}\\
&\qquad\times V^{s(n-k)+k+s}W_{p}^{n-k+1}U^k\Big)\\
&=\sum_{k=0}^{n}\mathfrak{S}_{s; h}(n, k|p, q)\Big(p^{n-k}q^{s(n-k)+k}V^{s(n-k)+k+1}W_{p}^{n-k}U^{k+1}+h[s(n-k)+k]_{p, q}\\
&\qquad\times V^{s(n-k+1)+k}W_{p}^{n-k+1}U^k\Big),
\end{align*}
where in the fourth equality we use Eq. \eqref{eq5} and in the last equality we use Eq. \eqref{eq3}.\\
Comparing the coefficients with Eq. \eqref{prdef} finishes the proof.
\end{proof}
\begin{cor}
The $(p, q)$-deformed generalized Stirling numbers $\mathfrak{S}_{s; h}(n, k|p, q)$ satisfy the following identity:
$$\mathfrak{S}_{s; h}(n, k|p, q)=h^{n-k}\mathfrak{S}_{s; 1}(n, k|p, q).$$
\end{cor}
\begin{proof}
The proof will be by induction on $n$. According to Eq. \eqref{Gstr}, we have:\\
For $n=0$, $\mathfrak{S}_{s; h}(0, 0|p, q)=0=\mathfrak{S}_{s; 1}(0, 0|p, q)$.\\
For $n=1$, $\mathfrak{S}_{s; h}(1, 0|p. q)=0=h\mathfrak{S}_{s; 1}(1, 0|p, q)$ and $\mathfrak{S}_{s; h}(1, 1|p, q)=1=\mathfrak{S}_{s; 1}(1, 1|p, q)$.\\
Now, using induction hypothesis and Eq. \eqref{recst}, one obtains
\begin{align*}
h^{n-k}\mathfrak{S}_{s; 1}(n, k|p, q)&=h^{n-k}\bigg(p^{n-k}q^{s(n-k)+k-1}\mathfrak{S}_{s; 1}(n-1, k-1|p, q)\\
&\qquad\qquad+[s(n-1-k)+k]_{p, q}\mathfrak{S}_{s; 1}(n-1, k|p, q)\bigg)\\
&=p^{n-k}q^{s(n-k)+k-1}h^{(n-1)-(k-1)}\mathfrak{S}_{s; 1}(n-1, k-1|p, q)\\
&\qquad\qquad+h\left([s(n-1-k)+k]_{p, q}h^{n-1-k}\mathfrak{S}_{s; 1}(n-1, k|p,  q)\right)\\
&=p^{n-k}q^{s(n-k)+k-1}\mathfrak{S}_{s; h}(n-1, k-1|p, q)\\
&\qquad\qquad+h[s(n-1-k)+k]_{p, q}\mathfrak{S}_{s; h}(n-1, k|p, q)\\
&=\mathfrak{S}_{s; h}(n, k|p, q).
\end{align*}
\end{proof}
\begin{rem}
For $s=0$ and $h=1$, the recurrence formula \eqref{recst} becomes
$$\mathfrak{S}_{0; 1}(n+1, k|p, q)=p^{n-k+1}q^{k-1}\mathfrak{S}_{0; 1}(n, k-1|p, q)+[k]_{p, q}\mathfrak{S}_{0; 1}(n, k|p, q),$$
which is the recurrence formula for the $(p, q)$-deformed Stirling numbers of the second kind $S_{p, q}(n, k)$, introduced  recently by the present author in \cite{LO}. Therefore,
$$\mathfrak{S}_{0; 1}(n, k|p, q)=S_{p, q}(n, k).$$
Moreover, for $p=1$ we obtain the well known recursive formula for the $q$-Stirling numbers of the second kind (see \cite{TMMSMS}). In addition, for $p=q=1$, the recurrence formula \eqref{recst} reduces to the one obtained by Mansour et al. for the generalized Stirling numbers of the second kind \cite{TMMS2011, TMMSMS2011}.
\end{rem}
\subsection{\texorpdfstring{The $(p, q)$-deformed generalized Touchard polynomials}{}}\label{subsec2}
In this subsection we establish a $(p, q)$-version of the generalized Touchard polynomials introduced in \cite{TMMS2013}, and some related properties. 
\begin{df}\label{defpqtch}
For $m\in{\R}$ $(n\in{\N}_{0})$, we define the $n$th $(p, q)$-deformed generalized Touchard polynomials for order $m$ by the formula
\begin{equation}\label{pqgtchrd}
T^{(m)}_{n; p, q}(x)=E_{p, q}(-p^{n}x)(X^mD_{p, q})^n e_{p, q}(x).
\end{equation} 
\end{df}
\begin{rem}
Obviously, if $p=1$, Eq. \eqref{pqgtchrd} reduces to the $q$-generalized Touchard polynomials of order $m$ \eqref{qgtchrd}. If further $q=p=1$, we obtain the classical generalized Touchard polynomials \eqref{ctchrd}.
\end{rem}
\begin{exm}
\begin{description}
\item[For $m=0$] $$T_{n, p; q}^{(0)}(x)=E_{p, q}(-x)D_{p, q}^{n}e_{p, q}(x)=p^{\binom{n}{2}} \text{ for all } n\in{\N}_{0}.$$
Note that as $p\rightarrow 1$, it approaches $q$-deformed case, where $T_{n; q}^{0}(x)=1$, for all $n\in{\N}_{0}$.
\item[For $m=1$]
\begin{align*}
T_{n; p, q}^{(1)}(x)&=E_{p, q}(-p^nx)(XD_{p, q})^ne_{p, q}(x)\\
&=E_{p, q}(-p^nx)\sum_{k=0}^{n}S_{p, q}(n, k)X^kN_{p}^{n-k}D_{p, q}^{k}e_{p, q}(x)\\
&=E_{p, q}(-p^nx)\sum_{k=0}^{n}p^{\binom{k}{2}}S_{p, q}(n, k)X^kN_{p}^{n-k}e_{p, q}(p^kx)\\
&=E_{p, q}(-p^nx)\sum_{k=0}^{n}\tilde{S}_{p, k}(n, k)x^ke_{p, q}(p^nx)\\
&=\sum_{k=1}^{n}\tilde{S}_{p, q}(n, k)x^k\\
&=\tilde{B}_{n}(p, q; x),
\end{align*}
where in the second equality we use Proposition \ref{propo},  and $\tilde{S}_{p, q}(n, k)=p^{\binom{k}{2}}S_{p, q}(n, k)$ are the other version  of $(p, q)$-deformed Stirling numbers of the second kind introduced in \cite{LO}, and $\tilde{B}_{n}(p, q; x)$ are the related $(p, q)$-Bell polynomials.
\end{description}
\end{exm}
\begin{prop}\label{prop311}
Let $m\in{\R}\setminus\{0\}$. Then, 
\begin{equation}\label{propo2}
D_{p, q}X^{m}=q^mX^mD_{p, q}+[m]_{p, q}X^{m-1}N_{p}.
\end{equation}
Moreover, 
\begin{equation}\label{prop21}
(X^mD_{p, q})^n=X^{n(m-1)}\sum_{k=0}^{n}\mathfrak{S}_{\frac{m-1}{m}; [m]_{p, q}}(n, k|p, q^m)X^kN_{p}^{n-k}D_{p, q}^{k}.
\end{equation}
\end{prop}
\begin{proof}
To prove Eq. \eqref{propo2}, we use the $(p, q)$-Leibniz formula \eqref{pqlf}:
\begin{align*}
D_{p, q}(X^mf(x))&=D_{p, q}(x^mf(x))\\
&=[m]_{p, q}x^{m-1}f(px)+(qx)^mD_{p, q}f(x)\\
&=[m]_{p, q}x^{m-1}N_{p}f(x)+(qx)^mD_{p, q}f(x)\\
&=[m]_{p, q}X^{m-1}N_{p}f(x)+q^mX^mD_{p, q}f(x).
\end{align*}
On the other hand, by recalling $X^{m-1}=(X^m)^{\frac{m-1}{m}}$, we can rewrite Eq. \eqref{propo2} as follows
$$D_{p, q}X^m=q^mX^mD_{p, q}+[m]_{p, q}(X^m)^{\frac{m-1}{m}}N_{p}.$$
Hence, by considering $D_{p, q}=U, X^m = V, N_{p}= W_{p}$, we have the following commutation relation for variables $U, V$ and $W_p$
$$UV=q^mVU+[m]_{p, q}V^{\frac{m-1}{m}}W_{p}.$$
Choosing $s=\frac{m-1}{m}, h=[m]_{p, q}$ and using Eq. \eqref{Gstr}, we obtain
$$(VU)^n=\sum_{k=0}^{n}\mathfrak{S}_{\frac{m-1}{m}; [m]_{p, q}}(n, k|p, q^m)V^{\frac{m-1}{m}(n-k)+k}W_{p}^{n-k}U^k.$$
Therefore, 
$$(X^mD_{p, q})^n=X^{n(m-1)}\sum_{k=0}^{n}\mathfrak{S}_{\frac{m-1}{m}; [m]_{p, q}}(n, k|p, q^m)X^{k}N_{p}^{n-k}D_{p, q}^k.$$
\end{proof}
\begin{rem}
Since $\mathfrak{S}_{0; 1}(n, k|p, q)=S_{p, q}(n, k)$, the formula \eqref{prop21} reduces to \eqref{meq} for $m=1$.
\end{rem}
\begin{thm}\label{vv}
For $m\in{\R}\setminus\{0\}$, the $(p, q)$-deformed generalized Touchard polynomials of order $m$ are given by the following formula:
\begin{equation}\label{pqtf}
T_{n; p, q}^{(m)}(x)=x^{n(m-1)}\sum_{k=0}^{n}\tilde{\mathfrak{S}}_{\frac{m-1}{m}; [m]_{p, q}}(n, k|p, q^m)x^k=x^{n(m-1)}\tilde{\mathfrak{B}}_{\frac{m-1}{m}; [m]_{p, q}|n; p, q^m}(x),
\end{equation}
where $\tilde{\mathfrak{S}}_{s; h}(n, k|p, q)=p^{\binom{k}{2}}\mathfrak{S}_{s; h}(n, k|p, q)$ is one possible definition of the $(p, q)$-deformed generalized Stirling numbers with the associated $(p, q)$-deformed generalized Bell polynomials $\tilde{\mathfrak{B}}_{s; h| n; p, q}(x):=\sum\limits_{k=0}^{n}\tilde{\mathfrak{S}}_{s; h}(n, k|p, q)x^k$.
\end{thm}
\begin{proof}
According to Definition \ref{defpqtch} and using Eqs. \eqref{expid}, \eqref{pqdex} and \eqref{prop21}, one obtains
\begin{align*}
T_{n; p, q}^{(m)}(x)&=E_{p, q}(-p^nx)(X^mD_{p, q})^ne_{p, q}(x)\\
&=E_{p, q}(-p^nx)X^{n(m-1)}\sum_{k=0}^{n}\mathfrak{S}_{\frac{m-1}{m}; [m]_{p, q}}(n, k|p, q^m)X^kN_{p}^{n-k}D_{p, q}^{k}e_{p, q}(x)\\
&=E_{p, q}(-p^nx)X^{n(m-1)}\sum_{k=0}^{n}\mathfrak{S}_{\frac{m-1}{m}; [m]_{p, q}}(n, k|p, q^m)X^kN_{p}^{n-k}p^{\binom{k}{2}}e_{p, q}(p^kx)\\
&=E_{p, q}(-p^nx)X^{n(m-1)}\sum_{k=0}^{n}p^{\binom{k}{2}}\mathfrak{S}_{\frac{m-1}{m}; [m]_{p, q}}(n, k|p, q^m)X^ke_{p, q}(p^nx)\\
&=x^{n(m-1)}\sum_{k=0}^{n}p^{\binom{k}{2}}\mathfrak{S}_{\frac{m-1}{m}; [m]_{p, q}}(n, k|p, q^m)x^k,
\end{align*}
which shows the assertion.
\end{proof}
In the next proposition, we establish a $(p, q)$-analogue of the recurrence relations \eqref{rctch} and \eqref{rqtch}   satisfied by the classical Touchard polynomials and by the $q$-deformed generalized Touchard polynomials, respectively. 
\begin{prop}
The $(p, q)$-deformed generalized Touchard polynomials $T_{n; pq}^{(m)}(x)$ satisfy the recursive formula:
\begin{equation}\label{pqrect}
T_{n+1; p, q}^{(m)}(x)=x^m\Big(N_{p}+E_{p, q}(-p^{n+1}x)e_{p, q}(p^nqx)D_{p, q}\Big)T_{n; p, q}^{(m)}(x).
\end{equation}
\end{prop}
\begin{proof}
According to Definition \ref{defpqtch}, we have
$$T_{n+1; p, q}^{(m)}(x)=E_{p, q}(-p^{n+1}x)(X^mD_{p, q})(X^mD_{p, q})^ ne_{p, q}(x).$$
On the other hand, 
$$(X^mD_{p, q})^ne_{p, q}(x)=e_{p, q}(p^nx)T_{n; p, q}^{(m)}(x).$$
Then, 
\begin{align*}
T_{n+1; p, q}^{(m)}(x)&=E_{p, q}(-p^{n+1}x)X^m\Big(D_{p, q}\Big(e_{p, q}(p^nx)T_{n; p, q}^{(m)}(x)\Big)\Big)\\
&=E_{p, q}(-p^{n+1}x)X^m\Big(e_{p, q}(p^{n+1}x)T_{n; p, q}^{(m)}(px)+e_{p, q}(p^nqx)D_{p, q}T_{n; p, q}^{(m)}(x)\Big)\\
&=x^mT_{n; p, q}^{(m)}(px)+x^mE_{p, q}(-p^{n+1}x)e_{p, q}(p^nqx)D_{p, q}T_{n; p, q}^{(m)}(x)\\
&=x^m\Big(N_{p}+E_{p, q}(-p^{n+1}x)e_{p, q}(p^nqx)D_{p, q}\Big)T_{n; p, q}^{(m)}(x),
\end{align*}
where in the second equality, we use the $(p, q)$-Leibniz rule \eqref{pqlf}.
\end{proof}
\begin{rem}
Clearly, for $p=1$, the recurrence $\eqref{pqrect}$ reduces to \eqref{rqtch}, and for $p=q=1$, it reduces to \eqref{rctch}.
\end{rem}
Now we are in the position to derive a $(p, q)$-analogue of the generalized Spivey's formula.
\begin{thm}\label{mainthm}
Let $m\in\mathbb{R}\setminus \{0\}$. Then, we have the following $(p, q)$-deformed generalized Spivey's relation:
\begin{multline}\label{pqgsp}
\tilde{\mathfrak{B}}_{\frac{m-1}{m}; [m]_{p, q}|n
+l; p, q^m}(1)=\sum_{k=0}^{n}\sum_{j=0}^{l}\Bigg(\begin{bmatrix}
n \\
n-k 
\end{bmatrix}_{(p^{m-1}, q^{m-1}, h_{m, j+(m-1)l}(p, q))}p^{(m-1)\left((n-k)(1+k)+kl\right)} \\ \times q^{k(j+l(m-1))}\tilde{\mathfrak{S}}_{\frac{m-1}{m}; [m]_{p, q}}(l, j|p, q^m)\Bigg)\tilde{\mathfrak{B}}_{\frac{m-1}{m}; [m]_{p, q}|k;  p, q^m}(p^{n+l-k})
\end{multline}
\end{thm}
Let us first prove the following two lemmas which will be the main ingredients for the proof of Theorem \ref{mainthm}.
\begin{lemma}\label{mainlem1}
For any positive integer $n, s\in\mathbb{R}$ and $m\in\mathbb{R}\setminus \{0\}$, we have
\begin{align}
(X^mD_{p, q})^nX^s&=X^s\left([s]_{p, q}X^{m-1}N_{p}+q^s(X^mD_{p, q})\right)^n \label{eq01}\\
&=\sum_{k=0}^{n}\begin{bmatrix}
n \\
n-k 
\end{bmatrix}_{(p^{m-1}, q^{m-1}, h_{m,s}(p, q))}[s]_{p, q}^{n-k}X^s(X^{m-1}N_{p})^{n-k}q^{ks}(X^mD_{p, q})^k, \label{eq02}
\end{align}
where $h_{m, s}(p, q):=\frac{q^s[m-1]_{p, q}}{[s]_{p, q}p^{m-1}}$.
\end{lemma}
\begin{proof}
To prove Eq. \eqref{eq01}, we use induction on $n$. For $n=1$ and by virtue of Proposition \ref{prop311}, Eq. \eqref{propo2}, one obtains
\begin{align*}
(X^mD_{p, q})X^s&=X^m\left([s]_{p, q}X^{s-1}N_{p}+q^s(X^sD_{p, q})\right)\\
&=X^s\left([s]_{p, q}X^{m-1}N_{p}+q^s(X^mD_{p, q})\right).
\end{align*}
Moreover, 
\begin{align*}
(X^mD_{p, q})^nX^s&=(X^mD_{p, q})(X^mD_{p, q})^{n-1}X^s\\
&=(X^mD_{p, q})\bigg(X^s\Big([s]_{p, q}X^{m-1}N_{p}+q^s(X^mD_{p, q})\Big)^{n-1}\bigg)\\
&=X^s\big([s]_{p, q}X^{m-1}N_{p}+q^s(X^mD_{p, q})\big)\big([s]_{p, q}X^{m-1}N_{p}+q^s(X^mD_{p, q})\big)^{n-1}\\
&=X^s\big([s]_{p, q}X^{m-1}N_{p}+q^s(X^mD_{p, q})\big)^n.
\end{align*}
For the second Eq. \eqref{eq02} we can use the $(q, h)$-binomial formula due to Benaoum \cite{HBB}. Let us recall the $(q, h)$-binomial formula from \cite{HBB}; for two variables $R$ and $S$ which satisfy $RS-qSR=hS^2$, one has
$$(R+S)^n=\sum_{k=0}^{n}\begin{bmatrix}
n \\
n-k 
\end{bmatrix}_{(q, h)}S^{n-k}R^k,$$
where $\begin{bmatrix}
n \\
n-k 
\end{bmatrix}_{(q, h)}:=\begin{bmatrix}
n \\
n-k 
\end{bmatrix}_{q}\prod_{j=0}^{k-1}(1+h[j]_{q})$ is the $(q, h)$-binomial coefficients. \\
Consider $S=[s]_{p, q}X^{m-1}N_{p}$ and $R=q^s(X^mD_{p, q})$, then 
\begin{align*}
(RS)f(x)&=q^s[s]_{p, q}X^mD_{p, q}X^{m-1}N_{p}f(x)\\
&=q^s[s]_{p, q}X^mD_{p, q}X^{m-1}f(px)\\
&=q^s[s]_{p, q}[m-1]_{p, q}X^mX^{m-2}f(p^2x)+q^s[s]_{p, q}q^{m-1}X^mX^{m-1}D_{p, q}f(px)\\
&=q^s[s]_{p, q}[m-1]_{p, q}X^{2m-2}N_{p}^{2}f(x)+q^s[s]_{p, q}q^{m-1}pX^mX^{m-1}N_{p}D_{p, q}f(x)\\
&=\frac{q^s[m-1]_{p, q}}{[s]_{p, q}p^{m-1}}S^2f(x)+\left(\frac{q}{p}\right)^{m-1}(SR)f(x).
\end{align*}
Therefore, 
$$RS=h_{m, s}(p, q)S^2+\left(\frac{q}{p}\right)^{m-1}SR.$$
Let us introduce the $(p,q, h)$-binomial coefficients $\begin{bmatrix}
n \\
n-k 
\end{bmatrix}_{(p, q, h)}$ as follows:
$$\begin{bmatrix}
n \\
n-k 
\end{bmatrix}_{(p, q, h)}:=\begin{bmatrix}
n \\
n-k 
\end{bmatrix}_{(\frac{q}{p}, h)}=p^{k(k-1)}\begin{bmatrix}
n \\
n-k 
\end{bmatrix}_{p, q}\prod_{j=0}^{k-1}(1+hp^{1-n}[j]_{p, q}).$$
Clearly, for $p=1, \begin{bmatrix}
n \\
n-k 
\end{bmatrix}_{(1, q, h)}=\begin{bmatrix}
n \\
n-k 
\end{bmatrix}_{(q, h)}.$\\
Thus,
\begin{align*}
\Big([s]_{p, q}X^{m-1}N_{p}+q^s(X^mD_{p, q})\Big)^n &=\sum_{k=0}^{n}\begin{bmatrix}
n \\
n-k 
\end{bmatrix}_{(p^{m-1}, q^{m-1}, h_{m, s}(p, q))}[s]_{p, q}^{n-k}(X^{m-1}N_{p})^{n-k}\\
& \hspace{2cm} \times q^{sk}(X^mD_{p, q})^k, 
\end{align*}
and this proves Eq. \eqref{eq02}.
\end{proof}
\begin{lemma}\label{mainlem2}
The $(p, q)$-generalized Touchard polynomials of order $m\in\mathbb{R}\setminus \{0\}$ satisfy the recursive formula
\begin{multline}\label{lemtchrd}
T_{n+l, p, q}^{(m)}(x)=\sum_{k=0}^{n}\sum_{j=0}^{l}\begin{bmatrix}
n \\
n-k 
\end{bmatrix}_{(p^{m-1}, q^{m-1}, h_{m, s}(p, q))}q^{k(j+(m-1)l)}p^{(m-k)(m-1)}\Bigl[j+l(m-1)\Bigr]_{p, q}^{n-k}\\ \times \tilde{\mathfrak{S}}_{\frac{m-1}{m}; [m]_{p, q}}(l, j|p, q^m)x^{j+(m-1)(l+n-k)}T_{k, p, q}^{(m)}(p^{n-k+l}x).
\end{multline}
\end{lemma}
\begin{proof}
We have 
\begin{align*}
\left(X^mD_{p, q}\right)^{n+l}&=\left(X^mD_{p, q}\right)^n\left(X^mD_{p, q}\right)^l\\
&=\left(X^mD_{p, q}\right)^nX^{l(m-1)}\sum_{j=0}^{l}\mathfrak{S}_{\frac{m-1}{m}; [m]_{p, q}}(l, j|p, q^m)X^jN_{p}^{l-j}D_{p, q}^{j}\\
&=\sum_{k=0}^{n}\sum_{j=0}^{l}\begin{bmatrix}
n \\
n-k 
\end{bmatrix}_{(p^{m-1}, q^{m-1}, h_{m, s}(p, q))}[j+l(m-1)]_{p, q}q^{k(j+l(m-1))}\\
&\qquad\quad \times \mathfrak{S}_{\frac{m-1}{m}; [m]_{p, q}}(l, j|p, q^m)X^{j+l(m-1)}(X^{m-1}N_{p})^{n-k}(X^mD_{p, q})^kN_{p}^{l-j}D_{p, q}^{j}\\
&=\sum_{k=0}^{n}\sum_{j=0}^{l}\begin{bmatrix}
n \\
n-k 
\end{bmatrix}_{(p^{m-1}, q^{m-1}, h_{m, s}(p, q))}p^{(n-k)(m-1)}[j+l(m-1)]_{p, q}q^{k(j+l(m-1))}\\
&\qquad\quad \times \mathfrak{S}_{\frac{m-1}{m}; [m]_{p, q}}(l, j|p, q^m)X^{j+(l+n-k)(m-1)}N_{p}^{n-k}(X^mD_{p, q})^kN_{p}^{l-j}D_{p, q}^{j}, 
\end{align*}
where in the second equality, we use Proposition \ref{prop311}, in the third equality, we use Lemma \ref{mainlem1} with $s=j+l(m-1)$, and in the last equality we use \eqref{cr2}.\\
Applying the above identity to the $(p, q)$-exponential function $e_{p, q}(x)$, yields 
\begin{align*}
(X^mD_{p, q})^{l+n}e_{p, q}(x)&=\sum_{k=0}^{n}\sum_{j=0}^{l}\begin{bmatrix}
n \\
n-k 
\end{bmatrix}_{(p^{m-1}, q^{m-1}, h_{m, s}(p, q))}p^{(n-k)(m-1)+\binom{j}{2}}\mathfrak{S}_{\frac{m-1}{m}; [m]_{p, q}}(l, j|p, q^m)\\
&\quad\times [j+l(m-1)]_{p, q} q^{k(j+l(m-1))}X^{j+(l+n-k)(m-1)}N_{p}^{n-k}(X^mD_{p, q})^ke_{p, q}(p^lx)\\
&=\sum_{k=0}^{n}\sum_{j=0}^{l}\begin{bmatrix}
n \\
n-k 
\end{bmatrix}_{(p^{m-1}, q^{m-1}, h_{m, s}(p, q))}p^{(n-k)(m-1)+\binom{j}{2}}\mathfrak{S}_{\frac{m-1}{m}; [m]_{p, q}}(l, j|p, q^m)\\
&\quad \times [j+l(m-1)]_{p, q} q^{k(j+l(m-1))} X^{j+(l+n-k)(m-1)}N_{p}^{n-k}\frac{T_{k, p, q}^{(m)}(p^lx)}{E_{p, q}(-p^{k+l}x)}\\
&=\sum_{k=0}^{n}\sum_{j=0}^{l}\begin{bmatrix}
n \\
n-k 
\end{bmatrix}_{(p^{m-1}, q^{m-1}, h_{m, s}(p, q))}p^{(n-k)(m-1)+\binom{j}{2}}\mathfrak{S}_{\frac{m-1}{m}; [m]_{p, q}}(l, j|p, q^m)\\
&\quad \times [j+l(m-1)]_{p, q} q^{k(j+l(m-1))} x^{j+(l+n-k)(m-1)}\frac{T_{k, p, q}^{(m)}(p^{n-k+l}x)}{E_{p, q}(-p^{n+l}x)}.
\end{align*}
Multiplying both sides of the above identity by $E_{p, q}(-p^{n+l}x)$ gives
\begin{multline*}
T_{n+l, p, q}^{(m)}(x)=\sum_{k=0}^{n}\sum_{j=0}^{l}\begin{bmatrix}
n \\
n-k 
\end{bmatrix}_{(p^{m-1}, q^{m-1}, h_{m, s}(p, q))}p^{(n-k)(m-1)}[j+l(m-1)]_{p, q}q^{k(j+l(m-1))}\\ \times \tilde{\mathfrak{S}}_{\frac{m-1}{m}; [m]_{p, q}}(l, j|p, q^m)x^{j+(l+n-k)(m-1)}T_{k, p, q}^{(m)}(p^{n-k+l}x).
\end{multline*}
\end{proof}
\begin{proof}[Proof of Theorem \ref{mainthm}]{\ \\}
By virtue of Lemma \ref{mainlem2} and Theorem \ref{vv}, one obtains the recurrence relation for $(p, q)$-deformed generalized Bell polynomials as follows
\begin{multline}
\tilde{\mathfrak{B}}_{\frac{m-1}{m}; [m]_{p, q}|n+l; p, q^m}(x)=\sum_{k=0}^{n}\sum_{j=0}^{l}\begin{bmatrix}
n \\
n-k 
\end{bmatrix}_{(p^{m-1}, q^{m-1}, h_{m, s}(p, q))}p^{(m-1)((n-k)(1+k)+kl)}\\ \times [j+l(m-1)]_{p, q}q^{k(j+l(m-1))}\tilde{\mathfrak{S}}_{\frac{m-1}{m}; [m]_{p, q}}(l, j|p, q^m)x^j\tilde{\mathfrak{B}}_{\frac{m-1}{m}; [m]_{p, q}|k; p, q^m}(p^{n+l-k}x).
\end{multline}
Plugging $x=1$ into this identity gives the assertion.
\end{proof}
\begin{rem} 
If $m=1$, the recurrence relation \eqref{pqgsp} reduces to the $(p, q)$-deformed Spivey's formula obtained by the present author \cite{LO}, and 
if $p=1$, one obtains the recurrence relation for $q$-deformed generalized Spivey's formula given by Mansour and Schork \cite{TMMS}.

\end{rem}
In the next result we present an analogue of  Dobiński's formula for $(p, q)$-deformed generalized Bell polynomials $\tilde{\mathfrak{B}}_{s; h|n; p, q}(x)$.
\begin{prop}[Dobiński's formula]
The $(p, q)$-deformed generalized Bell polynomials $\tilde{\mathfrak{B}}_{\frac{m-1}{m}; [m]_{p, q}|n; p, q^m}(x)$ satisfy the following relation:
\begin{equation}\label{pqdfdob}
\tilde{\mathfrak{B}}_{\frac{m-1}{m}; [m]_{p, q}|n; p, q^m}(x)=\frac{1}{e_{p, q}(p^nx)}\sum_{k\geq 0}\frac{\left(\prod\limits_{j=0}^{n-1}[k+j(m-1)]_{p, q}\right)}{[k]_{p, q}!}p^{\binom{k}{2}}x^k,
\end{equation}
and consequently, the $(p, q)$-deformed generalized Bell numbers $\tilde{\mathfrak{B}}_{\frac{m-1}{m}; [m]_{p, q}}(n|p, q^m)$ satisfy
\begin{equation}\label{pqdfdob2}
\tilde{\mathfrak{B}}_{\frac{m-1}{m}; [m]_{p, q}}(n|p, q^m)=\frac{1}{e_{p, q}(p^n)}\sum_{k\geq 0}\frac{\left(\prod\limits_{j=0}^{n-1}[k+j(m-1)]_{p, q}\right)}{[k]_{p, q}!}p^{\binom{k}{2}}.
\end{equation}
\end{prop}
\begin{proof}
Since $\left(X^mD_{p, q}\right)^nx^k=\left(\prod\limits_{j=0}^{n-1}[k+j(m-1)]_{p, q}\right)x^{k+n(m-1)}$, we have 
$$E_{p, q}(-p^nx)\left(X^mD_{p, q}\right)^ne_{p, q}(x)=E_{p, q}(-p^nx)\sum_{k\geq 0}\frac{\left(\prod\limits_{j=0}^{n-1}[k+j(m-1)]_{p, q}\right)}{[k]_{p, q}!}p^{\binom{k}{2}}x^{k+n(m-1)}.$$
Then, by Definition \ref{defpqtch}
$$T_{n, p, q}^{(m)}(x)=E_{p, q}(-p^nx)\sum_{k\geq 0}\frac{\left(\prod\limits_{j=0}^{n-1}[k+j(m-1)]_{p, q}\right)}{[k]_{p, q}!}p^{\binom{k}{2}}x^{k+n(m-1)}.$$
Hence,  Theorem \ref{vv} yields
$$\tilde{\mathfrak{B}}_{\frac{m-1}{m}; [m]_{p, q}|n; p, q^m}(x)=\frac{1}{e_{p, q}(p^nx)}\sum_{k\geq 0}\frac{\left(\prod\limits_{j=0}^{n-1}[k+j(m-1)]_{p, q}\right)}{[k]_{p, q}!}p^{\binom{k}{2}}x^k.$$
Setting $x=1$, we arrive at \eqref{pqdfdob2}.
\end{proof}
\begin{rem}
Note that, for $m=1$, Eq. \eqref{pqdfdob} reduces to the $(p, q)$-deformed Dobiński formula given by the present author \cite{LO}, and further, for $p=q=m=1$, one obtains the ordinary Dobiński formula for Bell polynomials 
$$B_{n}(x)=\frac{1}{e}\sum_{k\geq 0}\frac{k}{k!}x^k.$$
Thus, Eq. \eqref{pqdfdob} can be considered as the $(p, q)$-deformed generalized Dobiński formula.
\end{rem}
\section{Concluding remarks}\label{sec4}
Let us mention that for $m=1$, the $(p, q)$-deformed Touchard polynomials $T_{n; p, q}^{(1)}(x)$ considered in this paper are different from those considered by Herscovici and Mansour in \cite{OHTM2017}, and by Kim et al. in \cite{TKOHTMSHR2016}.\\
It would be very interesting to find new connections of the $(p, q)$-deformed generalized Touchard polynomials $T_{n; p, q}^{(m)}(x)$ to well known combinatorial numbers and polynomials. In particular, it was shown in \cite{TMMS2013} that the undeformed  generalized Touchard polynomials of order $m=\frac{1}{2}$ can be  expressed in terms of Hermite polynomials
$$T_{n}^{(\frac{1}{2})}(x)=\left(\frac{i}{2}\right)^{n}H_{n}(-i\sqrt{x}).$$
On the other hand, Mansour and Schork \cite{TMMS} showed that the $q$-deformed generalized Touchard polynomials $T_{n; q}^{(m)}(x)$ are connected to $q$-deformed Laguerre polynomials and $q$-deformed Bessel polynomials, respectively, for $m=2$ and $m=-1$.\\
Therefore, it would be interesting to determine whether the above relations hold true for the $(p, q)$-deformed generalized Touchard polynomials $T_{(m)}^{n; p, q}(x)$ considered in this paper.\\
However, the $(p, q)$-deformed generalized Stirling numbers $\mathfrak{S}_{s; h}(n|p, q)$ should be established deeply, as well as several special choices for the parameters $s$ and $h$ should be exhibited explicitly. These and other questions warrant further study.
\section{Acknowledgment}
Research partially supported by the Austrian FWF Project P29355-N35 and by the Polish National Science Center (NCN) grant 2016/21/B/ST1/00628.\\
The author would like to thank Franz Lehner for careful reading  and helpful comments.

\end{document}